\documentclass[11pt,a4paper,twoside]{amsart}
\usepackage{amsmath}
\usepackage{amssymb}
\usepackage{amsthm}
\usepackage{mathrsfs}

\theoremstyle{plain}
\newtheorem{theorem}{Theorem}[section]

\newtheorem{lemma}[theorem]{Lemma}
\newtheorem{corollary}[theorem]{Corollary}

\theoremstyle{remark}
\newtheorem{remark}[theorem]{Remark}

\allowdisplaybreaks[4]

\usepackage{amsrefs} 
\usepackage{hyperref} 

\hypersetup{
    colorlinks=true, 
    linktoc= page,   
}

\numberwithin{equation}{section}

\newcommand{\R}{\mathbb{R}}
\newcommand{\Z}{\mathbb{Z}}

\newcommand{\F}{\mathcal{F}}

\newcommand{\I}{\infty}

\newcommand{\norm}[1]{\left\lVert #1\right\rVert}
\newcommand{\xLebn}[2]{\left\lVert #1 \right\rVert_{L^{#2}_x}}

\newcommand{\Lebn}[2]{\left\lVert #1 \right\rVert_{L^{#2}}}
\newcommand{\xTLebn}[3]{\left\lVert #1 \right\rVert_{L^{#2}_x L^{#3}_T}}
\newcommand{\TxLebn}[3]{\left\lVert #1 \right\rVert_{L^{#2}_T L^{#3}_x}}

\newcommand{\xtLebn}[3]{\left\lVert #1 \right\rVert_{L^{#2}_x L^{#3}_t}}
\newcommand{\txLebn}[3]{\left\lVert #1 \right\rVert_{L^{#2}_t L^{#3}_x}}

\newcommand{\Sobn}[2]{\left\lVert #1 \right\rVert_{H^{#2}}}

\newcommand{\Jbr}[1]{\left\langle #1 \right\rangle}

\def\({\left(}
\def\){\right)}
\def\<{\left\langle}
\def\>{\right\rangle}
\def\le{\leqslant}
\def\ge{\geqslant}

\def\d{{\partial}}

\def \f{\phi}

\def \l{\lambda}

\def \d{\delta}

\def \pa{\partial}

\def \s{\sigma}
\def \a{\alpha}
\def \b{\beta}

\def \t{\theta}

\def \P{\Phi}

\def \F{\mathcal{F}}

\def \ga{\gamma}

\newcommand{\be}{\beta}

\newcommand{\ka}{\kappa}

\newcommand{\la}{\lambda}
\newcommand{\ti}{\widetilde}

\begin{document}
\title[LWP for the higher-order generalized KdV type equation]
{Local well-posedness for the higher-order generalized KdV type equation with low-degree of nonlinearity}

\author[H. Miyazaki]{Hayato MIYAZAKI}
\address[]{Teacher Training Courses, Faculty of Education, Kagawa University, Takamatsu, Kagawa 760-8522, Japan}
\email{miyazaki.hayato@kagawa-u.ac.jp}

\keywords{higher-order KdV equation, well-posedness}
\subjclass[2010]{35A01, 35Q53}
\date{}

\begin{abstract}
This paper is concerned with the local well-posedness for the higher-order generalized KdV type equation with low-degree of nonlinearity. The equation arises as a non-integrable and lower nonlinearity version of the higher-order KdV equation. As for the lower nonlinearity model of the KdV equation, Linares, the author and Ponce \cite{LMP} prove the local well-posedness under a non-degenerate condition introduced by Cazenave and Naumkin \cite{CaNa}. 
In this paper, it turns out that the well-posedness result can be extended into the higher-order equation. We also give a lower bound for the lifespan of the solution. The lifespan depends on two quantities determined by the initial data.  
\end{abstract}

\maketitle

\section{Introduction}
In this paper we consider the Cauchy problem for the higher-order generalized Korteweg-de Vries (KdV) type equation
\begin{align}
	\begin{cases}
	\pa_t u + \pa_x^{2j+1} u \pm |u|^{\a}\pa_x^{2j-1} u =0,\; (t, x) \in \R^{+} \times \R, \\
	u(x,0) = u_0(x),\; x \in \R,
	\end{cases}
	\label{hkdv} \tag{HK}
\end{align}
where $u=u(x,t)$ is a complex-valued unknown function, $\a \in (0,1)$ and $j \in \Z^{+}$. 
The equation \eqref{hkdv} appears as a non-integrable and lower nonlinearity version of the higher order KdV equation
\begin{align}
	\pa_t u + \pa_x^{2j+1} u + c_j u \pa_x^{2j-1} u + P(u, \ldots, \pa_x^{2j-2} u) =0 \label{kdv}
\end{align}
corresponding to the KdV hierarchy introduced by Lax \cite{La},
where $c_j \in \R \setminus \{0\}$, $j \in \Z^+$ and $P$ is a certain polynomial (see also \cite{KPV3}). For instance, the KdV equation and the fifth order KdV equation can be described as
\begin{align}
	&{}\pa_t u + \pa_x^{3} u + u \pa_x u = 0, \label{3kdv} \\
	&{}\pa_t u + \pa_x^{5} u - 10u \pa_x^{3} u +30u^2 \pa_x u -20\pa_x u \pa_x^2 u =0, \label{5kdv}
\end{align}
respectively. These equation arise as various physical phenomena such as long wave propagating in a channel and the interaction effects between short and long waves.
The KdV equation \eqref{3kdv}
on the line and the torus has been studied in huge mathematical and physical literatures. We refer to chapter 7-8 in \cite{LP} where we can summarize a lot of results for sharp local and global well-posedness, stability of special solutions, existence of blow-up solutions, and so on.
In the fifth order KdV equation \eqref{5kdv}, Ponce \cite{P} prove the local well-posedness in $L^2$-based Sobolev space $H^s(\R)$, $s \ge 4$. Later on, 
Guo, Kwak and Kwon \cite{GKK}, and Kenig and Pilod \cite{KP1}, independently improve the local well-posedness for $s \ge 2$. 
For more higher equation \eqref{kdv}, Kenig, Ponce and Vega \cite{KPV3, KPV4} prove the local well-posedness in the weighted Sobolev spaces $H^s(\R)\cap L^2(\Jbr{x}^r dx)$ for some $s$, $r \in \Z^+$ under more general nonlinearity. Further, Kenig and Pilod \cite{KP2} establish the local well-posedness in $H^s$ on the line and the torus at sufficiently large $s$.

On the other hand, in \cite{LMP}, Linares, the author and Ponce consider the lower nonlinearity model of \eqref{3kdv} as follows:
\begin{align}
	\pa_t u + \pa_x^3 u  \pm |u|^{\a}\pa_x u = 0,\; \a \in (0,1). \label{gkdv}
\end{align}
We emphasize that it is difficult to prove the local well-posedness for \eqref{gkdv} in $H^s(\R)$ or $H^s(\R)\cap L^2(\Jbr{x}^r dx)$ as long as we only employ the contraction principle, because the nonlinearity of \eqref{gkdv} is not Lipschitz continuous in those spaces.
Nevertheless, in \cite{LMP}, they establish the local well-posedness for \eqref{gkdv} in an appropriate class under a non-degenerate condition
\[
	\inf_{x \in \R} \Jbr{x}^{m(\a)} |u_0(x)| >0
\]
introduced by Cazenave and Naumkin \cite{CaNa} (see \eqref{thm:15}).
Our purpose of this paper is to extend their result into the case of \eqref{hkdv}. The main result is the following:
\begin{theorem}[Local well-posedness] \label{thm:1}
Denote $m= \left[ \frac{1}{\a} \right]+1$. Let $s \in \Z^{+}$ satisfy $s - j+1 \ge 2jm+2j+2$. 
Assume that  
\begin{align}
	\begin{aligned}
	&{}u_0 \in H^s(\R), \quad \Jbr{x}^m u_0 \in L^{\I}(\R), \\
	&{}\Jbr{x}^m \pa_x^{\ga} u_0 \in L^{2}(\R), \quad \ga = 1, \cdots, 2j+2
	\end{aligned}
	\label{thm:14}
\end{align}
with 
\begin{align}
	&{}\Sobn{u_0}{s} + \Lebn{\Jbr{x}^m u_0}{\I} + \sum_{\ga =1}^{2j+2} \norm{\Jbr{x}^{m} \pa^{\ga}_{x} u_0}_{L^2} =:\d, 
	\label{thm:16}
\end{align}
and
\begin{align}
	\inf_{x \in \R} \Jbr{x}^m |u_0(x)| =: \l >0. \label{thm:15}
\end{align}
Then there exists $T=T(\d, \l; \a, s, j)>0$ such that \eqref{hkdv} has a unique local solution
\begin{align}
	u \in C([0,T], H^s(\R)),\quad \Jbr{x}^m u \in C([0,T], L^\I(\R)) \label{thm:12}
\end{align}
with
\begin{align}
	\begin{aligned}
	&{}\Jbr{x}^m \pa_x^{\ga} u \in C([0,T], L^2(\R)),\; \ga = 1, \ldots, 2j+2, \\
	&{}\pa_x^{s+j-l} u \in L^{\infty}(\R, L^{2}([0,T])),\; l = 0, 1, \ldots, j-1, 
	\end{aligned}
	\label{thm:13}
\end{align}
and
\begin{align*}
\sup_{0 \le t \le T} \Lebn{\Jbr{x}^m (u(t)-u_0)}{\I} \le \frac{\l}2.
\end{align*}
Moreover, the map $u_0 \mapsto u(t)$ is continuous in the following sense:
For any compact $I \subset [0,T]$, there exists a neighborhood $V$ of $u_0$ satisfying  \eqref{thm:14} and \eqref{thm:15} such that the map 
is Lipschitz continuous from $V$ into the class defined by \eqref{thm:12} and \eqref{thm:13}.
\end{theorem}

\begin{remark}
The non-degenerate condition \eqref{thm:15} comes from that $|u|^{\a}$ is not regular enough (only $C^{\a}$).
The condition \eqref{thm:15} enables us to carry out the contraction principle. 
This approach is firstly introduced to a nonlinear Schr\"odinger equation by \cite{CaNa}. Later on, it is applied to the derivative nonlinear Schr\"odinger equation with low-degree of nonlinearity by Linares, Ponce and Santos \cite{LPS, LPS2}, as well as \cite{LMP} for \eqref{gkdv}.
\end{remark}

\begin{remark}
We employ the Kato smoothing effect (Lemma \ref{lem:2} below) to remove derivative loss of the nonlinearity. 
This provides us the additional regularity $\pa_x^{s+j-l} u \in L^{\infty}(\R, L^{2}([0,T]))$ for any $0 \le l \le j-1$ in \eqref{thm:13}. 
\end{remark}

\begin{remark}
The regularity condition $s - j+1 \ge 2jm+2j+2$ arises from the estimate of the following norm:
\[
	\TxLebn{\Jbr{x}^m \pa_x^{2j+2}(|u|^{\a} \pa_x^{2j-1}u)}{1}{2}
\]
using the relation $e^{t\pa_x^{2j+1}} x^m e^{-t\pa_x^{2j+1}} = (x+(2j+1)t\pa_x^{2j})^m$ exhibited in Section \ref{sec:2}. In detail, see the above estimate of \eqref{est:in}.
\end{remark}

We here define the lifespan of the solution to \eqref{hkdv} by 
\begin{align*}
	T_{\d, \la} := \sup \{T \in (0, \I];\ &{} \text{there exists a unique solution to \eqref{hkdv}} \\
	&{} \qquad \qquad \text{in the class given by Theorem \ref{thm:1}} \},
\end{align*}
where $\d$ and $\l$ are defined by \eqref{thm:16} and \eqref{thm:15}.
Once the local well-posedness is established by the contraction principle, we have a lower bound estimate for the lifespan of the solution.
\begin{corollary}[Lower bound for the lifespan] \label{cor:1}
Let $\d$ be as in \eqref{thm:16}. Define $\l$ by \eqref{thm:15}.
Under the same assumptions as in Theorem \ref{thm:1}, there exists a constant $C \in (0,1)$ such that
\begin{align*}
	T_{\d,\l}^{\frac{1}{2j}} \ge{}& \frac{C \l}{\d \( 1 + \d^{\ka} \( \frac{1+\la}{\la} \)^{s-j+2-\a} \)},
\end{align*}
where $\kappa = s-j+2$ if $\d \ge 1$, otherwise $\kappa = \a$. 
\end{corollary}

\begin{remark}
Note that $\d > \la$.
In the case $j=1$, Corollary \ref{cor:1} was proven in \cite{M}. However it is necessary that the lower bound in \cite{M} is corrected slightly, because it is required to take an estimate as in \eqref{lsd:1} into account.
\end{remark}

The rest of the paper is organized as follows:
In Section \ref{sec:2}, we collect some estimates for the linear evolution operator and an interpolation inequality.
Section \ref{sec:3} is devoted to some nonlinear estimates playing a crucial role in Section \ref{sec:4}.
We finally prove Theorem \ref{thm:1} and Corollary \ref{cor:1} in Section \ref{sec:4}.

We here introduce several notations used throughout this paper.

\noindent
\textbf{Notations:}
We set $\Jbr{x} = (1+|x|^2)^{\frac12}$ and $[x]$ denotes the greatest integer less than or equal to $x$ for any $x \in \R$. 
For any $q$, $r \ge 1$, We denote $\txLebn{F}{q}{r} = \| \norm{F}_{L^r_x (\R)} \|_{L^q_t(\R)}$ and $\xtLebn{F}{r}{q} = \| \norm{F}_{L^q_t(\R)} \|_{L^r_x(\R)}$. 
We also define $\TxLebn{F}{q}{r} = \| \norm{F}_{L^r_x (\R)} \|_{L^q_t([0,T])}$ and $\xTLebn{F}{r}{q} = \| \norm{F}_{L^q_t([0,T])} \|_{L^r_x(\R)}$. $\F$ stands for the usual Fourier transform on $\R$. Let $U_j(t) = e^{-t \pa_x^{2j+1}}$ be the linear evolution operator defined by $U_{j}(t) = \F^{-1} e^{-t(i\xi)^{2j+1}} \F$ for any $j \in \Z^+$.

\section{Preliminary} \label{sec:2}

We start this section presenting some linear estimates. The first one is concerning the sharp Kato smoothing effect found in \cite{KPV2, KPV}.

\begin{lemma}[\cite{KPV2, KPV4}] \label{lem:2}
Let $j \in \Z^+$. The following estimates hold:
\begin{align*}
	&{} \xtLebn{\pa_x^{j} U_{j}(t) \f}{\I}{2} \le C \xLebn{\f}{2}, \\ 
	&{} \txLebn{\pa_x^{\s} \int_{\R} U_{j}(t-s) f(s) ds}{\I}{2} \le C T^{\a} \xtLebn{f}{p}{2}, \\ 
	&{} \xTLebn{\pa_x^{j+\s} \int_{\R} U_{j}(t-s) f(s) ds}{\I}{2} \le C T^{\a} \xTLebn{f}{p}{2} 
\end{align*} 
for any $\s=0, \cdots, j$, where $\a = (j-\s)/2j$ and $p= 2j/(j+\s)$.
\end{lemma}

The next one is an interpolation inequality used throughout this paper. 
\begin{lemma}[\cite{LMP}] \label{lem:0}
Let $\mu>0$ and $r\in\Z^+ $. Denote $\theta \in[0,1]$ with $(1-\theta)r\in\Z^+$. Then it holds that  
\begin{align*}
	\Lebn{\Jbr{x}^{\theta \mu} \pa_x^{(1-\theta)r} f}{2} \le{}& C\Lebn{\Jbr{x}^{\mu}  f}{2}^{\theta} \Lebn{\pa_x^{r} f}{2}^{1-\theta} +L.O.T., 
\end{align*}
where the lower order terms $\,L.O.T.$ are bounded by
\begin{align*}
\sum_{\substack{0 \le \beta \le 1, \\ (1-\beta)(r-1)\in\Z^+}} \Lebn{\Jbr{x}^{\beta (\mu-1)} \pa_x^{(1- \b)(r-1)} f}{2} . 
\end{align*}
\end{lemma}

\begin{proof}
We shall give a proof for self-containedness.
Let us only show the case $r$ is even ($r = 2N$), because the odd case is similar. 
We also only consider real-valued functions for simplicity.
Set
\begin{align*}
	A_l = \Lebn{\Jbr{x}^{\frac{r-l}{r}\mu }\pa_x^l f}{2}
\end{align*}
for any $0 \le l \le r$. As for $A_N$, combining $A_0$ with $A_{2N}$, it follows from the integration by parts that
\begin{align}
	\begin{aligned}
	&{}\int_{\R} \Jbr{x}^{\mu} f \pa_x^{2N} f\, dx \\
	\sim{}& \int_{\R} \Jbr{x}^{\mu} \(\pa_x^N f\)^2\, dx + \sum_{k=1}^{N} \int_{\R} \Jbr{x}^{\mu-k} \pa_x^{N-k} f \pa_x^{N} f\, dx.
	\end{aligned}
	\label{po:0}
\end{align}
Here we write $X \sim Y + Z$ to indicate $X = \pm Y + cZ$ for some constant $c$.
For the last term of R.H.S in the above, we estimate  
\begin{align*}
	&{}\int_{\R} \Jbr{x}^{\mu-k} \pa_x^{N-k} f \pa_x^{N} f\, dx \\
	\le{}& \Lebn{\Jbr{x}^{\(\mu-1\) \b_0}\pa_x^{(r-1)(1-\b_0)} f}{2} \Lebn{\Jbr{x}^{\(\mu-1\) \t_0}\pa_x^{(r-1)(1-\t_0)} f}{2}	,
\end{align*}
since it is possible to take $\b_0 = \b_0(k)$, $\t_0 \in [0,1]$ such that
\begin{align*}
	\begin{cases}
	\mu -k \le (\mu-1)(\b_0 + \t_0), \\ 
	N-k = (r-1)(1-\b_0),\; N = (r-1)(1-\t_0).
	\end{cases}
\end{align*}
Hence one sees from \eqref{po:0} and the H\"older inequality that
\begin{align}
	\begin{aligned}
	A_N^2 \le{}& A_0 A_r \\
	&{}+ C \sum_{k=1}^{N} \Lebn{\Jbr{x}^{\(\mu-1\) \b_0(k)}\pa_x^{(r-1)(1-\b_0(k))} f}{2} \\
		 &{} \quad \times \Lebn{\Jbr{x}^{\(\mu-1\) \t_0}\pa_x^{(r-1)(1-\t_0)} f}{2}.
	\end{aligned}
	\label{po:1}
\end{align}
Let us next estimate $A_{l}$ for all $1 \le l \le N-1$. Unifying $A_0$ and $A_{2l}$, by the integration by parts, we deduce that
\begin{align*}
	\int_{\R} \Jbr{x}^{\mu} f \Jbr{x}^{\frac{r-2l}{r} \mu} \pa_x^{2l} f\, dx \sim \int_{\R} \Jbr{x}^{\frac{2(r-l)}{r} \mu} ( \pa_x^l f )^2\, dx + \sum_{k=1}^{l} L_{l,k},
\end{align*}
where 
\[
	L_{l,k} = \int_{\R} \Jbr{x}^{\frac{2(r-l)}{r} \mu-k} \pa_x^{l-k} f \pa_x^{l} f\, dx.
\]
A use of integration by parts gives us
\begin{align*}
	L_{l,1} = \frac12 \int_{\R} \( \Jbr{x}^{\frac{(r-l)}{r} \mu-1} \pa_x^{l-1} f \)^2\, dx \le C \Lebn{\Jbr{x}^{\frac{r-l}{r-1}(\mu-1)}\pa_x^{l-1} f}{2}.
\end{align*}
We further obtain 
\begin{align*}
	L_{l,k} \le \Lebn{\Jbr{x}^{\(\mu-1\) \b_1}\pa_x^{(r-1)(1-\b_1)} f}{2} \Lebn{\Jbr{x}^{\(\mu-1\) \t_1}\pa_x^{(r-1)(1-\t_1)} f}{2},
\end{align*}
for any $k \ge 2$, since there exist $\b_1 = \b_1(k)$, $\t_1 \in [0,1]$ such that
\begin{align*}
	\begin{cases}
	\frac{2(r-l)}{r} \mu - k \le (\mu-1)(\b_1 + \t_1), \\ 
	l-k = (r-1)(1-\b_1),\; l = (r-1)(1-\t_1).
	\end{cases}
\end{align*}
This implies
\begin{align}
	\begin{aligned}
	A_l^2 \le{}& A_0 A_{2l} \\
	&{}+ C \sum_{k=1}^{l} \Lebn{\Jbr{x}^{\(\mu-1\) \b_1(k)}\pa_x^{(r-1)(1-\b_1(k))} f}{2} \\
		&{} \quad \times \Lebn{\Jbr{x}^{\(\mu-1\) \t_1}\pa_x^{(r-1)(1-\t_1)} f}{2}
	\end{aligned}
	\label{po:2}
\end{align}
for any $0 \le l \le N-1$.
Finally we shall consider $A_{l}$ for all $N+1 \le l \le r-1$. Combining $A_{2l-r}$ with $A_r$, using the integration by parts, we reach to
\begin{align*}
	\int_{\R} \Jbr{x}^{\frac{2(r-l)}{r} \mu} \pa_x^{2l-r} f \pa_x^r f\, dx \sim \int_{\R} \Jbr{x}^{\frac{2(r-l)}{r} \mu} ( \pa_x^l f )^2\, dx + \sum_{k=1}^{r-l} L_{l,k}.
\end{align*}
Just arguing as in \eqref{po:2}, one has
\begin{align}
	\begin{aligned}
	A_l^2 \le{}& A_{2l-r} A_{r} \\
	&{}+ C\sum_{k=1}^{r-l} \Lebn{\Jbr{x}^{\(\mu-1\) \b_1(k)}\pa_x^{(r-1)(1-\b_1(k))} f}{2} \\
	&{} \quad \times \Lebn{\Jbr{x}^{\(\mu-1\) \t_1}\pa_x^{(r-1)(1-\t_1)} f}{2}
	\end{aligned}
	\label{po:3}
\end{align}
for all $N+1 \le l \le r-1$. Collecting \eqref{po:1}, \eqref{po:2} and \eqref{po:3}, we have the desired estimate.
\end{proof}

We finish this section by stating properties of  $U_{j}(t)$. 
Combining the fact
\begin{align}
	U_{j}(-t) x U_{j}(t) = x +(2j+1) t \pa_x^{2j} \label{lin:1}
\end{align}
with Lemma \ref{lem:0}, the following is valid:
\begin{lemma} \label{lem:3}
Let $j \in \Z^+$. It holds that
\begin{align*}
	{}\Lebn{\Jbr{x}^{\b} U_j(t)f}{2} 
	\le{}& C \Jbr{t}^{\b} \( \Lebn{\Jbr{x}^{\b} f}{2} + \Lebn{\pa_x^{2j\b} f}{2}\)
\end{align*}
for any $\b \in \Z^{+}$. 
\end{lemma}

\begin{proof}
By using \eqref{lin:1}, we have
\[
U_j(-t) x^{\b} U_{j}(t)f = \( x + (2j+1) t \pa_x^{2j} \)^{\b} f
\]
for all $\b \in \Z^+$. This implies that
\begin{align*}
	{}\Lebn{\Jbr{x}^{\b} U_j(t)f}{2} 
	\le{}& C\( \Lebn{U_j(t)f}{2} + \Lebn{|x|^\b U_j(t)f}{2} \) \\
	\le{}& C\Lebn{f}{2} + C\Lebn{\( x + (2j+1) t \pa_x^{2j} \)^{\b} f}{2}.
\end{align*}
Set $\ga = 2j$.
A direct calculation gives us 
\begin{align*}
	&{}\( x + (2j+1) t \pa_x^{\ga} \)^{\b}f \\
	={}& x^\b f + t \sum_{k=0}^{\min(\ga, \b-1)} c_{1,k} x^{\b-1-k} \pa_x^{\ga-k} f  + t^2 \sum_{k=0}^{\min(2\ga, \b-2)} c_{2,k} x^{\b-2-k} \pa_x^{2\ga-k} f \\
	&{}\cdots  + t^{\b-1} \( c_{\b-1,0} x \pa_x^{(\b-1)\ga} f + c_{\b-1,1} \pa_x^{(\b-1)\ga-1} f \) + t^{\b} \pa_x^{\b \ga} f,
\end{align*}
where $c_{1, k}, \cdots, c_{\b-1, k} \in \Z^+$ depending on $j$. Applying Lemma \ref{lem:0} to $L^2$ norm of each terms repeatedly, we obtain the desired estimate.
\end{proof}

\section{Nonlinear estimates} \label{sec:3}

In this section, we collect nonlinear estimates. To this end, we introduce the key norm corresponding to the solution space in Theorem \ref{thm:1}:
\begin{align*}
	\norm{u}_{X_{T}} :={}& \norm{u}_{L^{\I}_T H^{s}} + \norm{\Jbr{x}^{m} u}_{L^{\I}_T L^\I_x} \\
	&{}+ \sum_{\ga =1}^{2j+2} \norm{\Jbr{x}^{m} \pa^{\ga}_{x} u}_{L^{\I}_T L^2_x} + \sum_{l=0}^{j-1} \xTLebn{\pa_x^{s+j-l} u}{\I}{2}
\end{align*} 
for all $s$, $j$, $m \in \Z^+$ and any $T >0$.

\begin{lemma} \label{lem:4}
Let $m= \left[ \frac{1}{\a} \right]+1$. 
Fix $s$, $j  \in \Z^+$ with $s-j+1 \in \Z^+$. Put $q \in (1, 2]$ and $\ga \in \Z^+$ with $\ga \le 2j+2$. Then it holds that
\begin{align}
&{}	\TxLebn{|u|^{\a} \pa_x^{2j-1} u}{1}{2} \le CT \norm{u}_{X_T}^{\a+1},
\label{lem:40} \\
&{}	\begin{aligned}
	&{}\xTLebn{\partial_x^{s-j+1} (|u|^{\a} \pa_x^{2j-1} u)}{q}{2} \\
	\le{}& C\norm{u}_{X_T}^{\a+1} + CT^{1/2} \la^{\a-1} \norm{u}_{X_T}^2 
	+ CT^{1/2} \la^{\a-(s-j+1)} \norm{u}_{X_T}^{s-j+2}
	\end{aligned}
\label{lem:42} \\
	&{}\TxLebn{\Jbr{x}^m \pa_x^{\ga} (|u|^{\a}\pa_x^{2j-1} u)}{1}{2} \le CT\la^{\a-1} \norm{u}_{X_T}^2 + CT \la^{\a-\ga} \norm{u}_{X_T}^{\ga+1}
\label{lem:43}
\end{align}
for any $T>0$ as long as $\Jbr{x}^m |u(t, x)| \ge \frac{\l}{2}$ for any $(t,x) \in [0,T] \times \R$.
\end{lemma}

\begin{proof}
To simplify the exposition, we shall consider real-valued functions. 
\eqref{lem:40} is immediate from Sobolev embedding. Let us prove \eqref{lem:42}. A use of the Leibniz rule gives us
\begin{align*}
	\xTLebn{\partial_x^{s-j+1} (|u|^{\a} \pa_x^{2j-1} u)}{q}{2} \le{}& 	\sum_{k=0}^{s-j+1} \xTLebn{\pa_x^k (|u|^{\a}) \pa_x^{s+j-k} u}{q}{2} \\
	=:{}& A_0 + \sum_{k=1}^{j-1} A_k + \sum_{k=j}^{s-j+1} A_k.
\end{align*}
By the H\"older inequality, we see that
\begin{align*}
	A_0 
	\le {}& C\xTLebn{\pa_x^{s+j}u}{\I}{2} \TxLebn{\Jbr{x} |u|^{\a}}{\I}{\I} \Lebn{\Jbr{x}^{-1}}{q} \\
	\le {}& C\xTLebn{\pa_x^{s+j}u}{\I}{2} \TxLebn{\Jbr{x}^{m} u}{\I}{\I}^{\a}.
\end{align*}
In terms of $A_k$ for $1 \le k \le s-j+1$, the elements can be written as
\begin{align*}
	\pa_x^k (|u|^{\a}) \pa_x^{s+j-k} u ={}& \sum_{n=1}^{k} \sum_{\substack{ \b_1 + \cdots + \b_n = k \\ 1 \le \beta_1, \cdots, \beta_n \le k}} C_{\vec{\beta}} |u|^{\alpha-2n}u^n \, \pa_x^{\beta_1}u \cdots \partial_x^{\beta_n}u\, \pa_x^{\b_0} u \\
	=:{}& \sum_{n=1}^{k} \sum_{\substack{ \b_1 + \cdots + \b_n = k \\ 1 \le \beta_1, \cdots, \beta_n \le k}} F_{\b_1,\ldots,\b_n},
\end{align*}
where $\b_0 = s+j-k$.
When $1 \le k \le j-1$, in light of $\b_0 \ge s+1$, using $\Jbr{x}^m |u(t, x)| \ge \frac{\l}{2}$,
we deduce from the H\"older and the Gagliardo-Nirenberg inequality that
\begin{align*}
&{}\xTLebn{F_{\b_1,\ldots,\b_n}}{q}{2} \\
	\le{}& C\l^{\a-n} \xTLebn{\Jbr{x}^{m(n-\a)} \pa_x^{\beta_1}u \cdots \partial_x^{\beta_n}u\, \pa_x^{\b_0} u}{q}{2} \\
	\le{}& C \l^{\a-n} \xLebn{\Jbr{x}^{-1}}{2} \prod_{i=1}^{n-1} \TxLebn{\Jbr{x}^{m}\, \pa_x^{\b_i}u}{\I}{\I} \\
	&{} \quad \times \TxLebn{\Jbr{x}^{m}\, \pa_x^{\b_n}u}{\I}{\frac{2q}{2-q}} \xTLebn{\pa_x^{\b_0}u}{\I}{2} \\
		\le{}& C \l^{\a-n} \prod_{i=1}^{n} \( \TxLebn{\Jbr{x}^{m}\, \pa_x^{\b_i}u}{\I}{2} + \TxLebn{\Jbr{x}^{m}\, \pa_x^{\b_i+1}u}{\I}{2} \) \\
	&{} \quad \times \xTLebn{\pa_x^{\b_0}u}{\I}{2} \\
	\le{}& C\l^{\a-n} \norm{u}_{X_T}^{n+1}.
\end{align*}
Note that $m\a \ge 1$.
In the case $j \le k \le s-j+1$, in view of $\b_i \le s$, the similar calculation shows
\begin{align*}
	\xTLebn{F_{\b_1,\ldots,\b_n}}{q}{2} \le{}& C\l^{\a-n} \xTLebn{\Jbr{x}^{m(n-\a)} \pa_x^{\beta_1}u \cdots \partial_x^{\beta_n}u\, \pa_x^{\b_0} u}{q}{2} \\
	\le{}& C T^{\frac12} \l^{\a-n}  \prod_{i=1}^{n-1} \TxLebn{\Jbr{x}^{m\t_i}\, \pa_x^{\b_i}u}{\I}{\I} \\
	&{} \times \TxLebn{\Jbr{x}^{m\t_n}\, \pa_x^{\b_n}u}{\I}{\frac{2q}{2-q}} \TxLebn{\Jbr{x}^{m\t_0}\, \pa_x^{\b_0}u}{\I}{2} \\
	\le{}& C T^{\frac12} \l^{\a-n} \prod_{i=1}^{n} \( \TxLebn{\Jbr{x}^{m\t_i}\, \pa_x^{\b_i}u}{\I}{2} + \TxLebn{\Jbr{x}^{m\t_i}\, \pa_x^{\b_i+1}u}{\I}{2} \) \\
	&{} \times \TxLebn{\Jbr{x}^{m\t_0}\, \pa_x^{\b_0}u}{\I}{2}.
\end{align*}
We remark that it is possible to choose $\t_i \in [0,1]$ satisfying 
\begin{align}
\begin{aligned}
	&{} r_0 \in \{1, \cdots , 2j+2 \}, \quad r_i \in \{1, \cdots , 2j+1 \}, \quad i =1, 2, \cdots, n, \\
	&{} \b_i  = (1-\t_i)s+ \t_i r_i, \quad i=0, 1, \cdots, n.
\end{aligned}
\label{A20}
\end{align}
By using Lemma \ref{lem:0} iteratively, we deduce that for any $0 \le i \le n$ and $\s =0$, $1$,
\begin{align*}
	&{} \TxLebn{\Jbr{x}^{m\t_i}\, \pa_x^{(1-\t_i)s+ \t_i r_i + \s}u}{\I}{2} \\
	\le{}& C\TxLebn{\Jbr{x}^m \pa_x^{r_i + \s} u}{\I}{2}^{\t_i} \TxLebn{\pa_x^s u}{\I}{2}^{1-\t_i} \\
	&{} + C \hspace{-1cm} \sum_{\substack{0 \le \eta_i \le 1, \\ (1-\eta_i)(s-1)+\eta_i r_i + \s  \in\Z^+}} \hspace{-1cm} \TxLebn{\Jbr{x}^{(m-1)\eta_i}\, \pa_x^{(1-\eta_i)(s-1)+ \eta_i r_i + \s}u}{\I}{2} \\
	\le{}& C \norm{u}_{X_T}.
\end{align*}
Combining the above, \eqref{lem:42} is valid. To complete the proof, it remains to verify 
\begin{align*}
	\t_0+ \t_1 + \cdots +\t_n \ge n-\a. 
\end{align*}
Collecting \eqref{A20}, together with $\b_1 + \cdots + \b_n = k$ and $\b_0 = s+j-k$, we see
\begin{align*}
	j = \( n- (\t_0+\t_1+ \cdots + \t_n) \)s + \t_0 r_0+\t_1r_1+ \cdots + \t_n r_n. 
\end{align*}
By means of $r_0$, $r_i \in \{1,\cdots, 2j+1 \}$, we obtain
\begin{align*}
	\t_0 + \t_1+ \cdots +\t_n ={}& n + \frac{\t_0 r_0 + \t_1 r_1+ \cdots +\t_n r_n - j}{s}\\
	\ge{}& n+\frac{\t_0+\t_1+ \cdots +\t_n- j}{s},
\end{align*}
Thanks to $s -1 \ge \frac{2}{\a}j +3j$, it is concluded that
\begin{align*}
	\t_0+ \t_1+ \cdots +\t_n \ge \frac{ns-j}{s-1} > n - \frac{1}{\frac{2}{\a}+3} >  n - \a.
\end{align*}

Finally we shall show \eqref{lem:43}.
By the Leibniz rule, it holds that
\begin{align*}
	\TxLebn{\Jbr{x}^m \partial_x^{\ga} (|u|^{\a} \pa_x^{2j-1} u)}{1}{2} \le{}& T \sum_{k=0}^{\ga} \TxLebn{\Jbr{x}^m \pa_x^k (|u|^{\a}) \pa_x^{2j-1+\ga-k} u}{\I}{2} \\
	=:{}& T\( \ti{A}_0 + \sum_{k=1}^{\ga} \ti{A}_k \)
\end{align*}
for any $\ga \le 2j+2$. We then deduce from Lemma \ref{lem:0} and $\Jbr{x}^m |u(t, x)| \ge \frac{\l}{2}$ that
\begin{align*}
	\ti{A}_0 
	\le{}& C\la^{\a-1} \TxLebn{\Jbr{x}^m u}{\I}{\I} \TxLebn{\Jbr{x}^{m-1} \pa_x^{2j-1+\ga}u}{\I}{2} \\
	\le{}& C\la^{\a-1} \TxLebn{\Jbr{x}^m u}{\I}{\I} \\
	&{}\times \( \TxLebn{\Jbr{x}^{m} \pa_x^{\ga-1}u}{\I}{2}^{1-\frac{1}{m}}\TxLebn{\pa_x^{2mj+\ga-1}u}{\I}{2}^{\frac{1}{m}} + L.O.T. \) \\
	\le{}& C\la^{\a-1} \TxLebn{\Jbr{x}^m u}{\I}{\I} \\
	&{}\times \( \TxLebn{\Jbr{x}^{m} \pa_x^{\ga-1}u}{\I}{2} + \TxLebn{\pa_x^{2mj+\ga-1}u}{\I}{2} + L.O.T. \).
\end{align*}
Here, by means of a successive use of Lemma \ref{lem:0}, one has
\begin{align*}
	L.O.T. \le{}& C \hspace{-35pt} \sum_{\substack{0 \le \eta \le 1, \\ (1-\eta)(2mj -1)+ (\ga-1)  \in\Z^+}} \hspace{-22pt}\TxLebn{\Jbr{x}^{(m-1)\eta}\, \pa_x^{(1-\eta)(2mj -1)+(\ga -1) }u}{\I}{2} \\
	\le{}& C \( \TxLebn{\Jbr{x}^{m-1} \pa_x^{\ga-1}u}{\I}{2} + \norm{u}_{L^{\infty}_T \dot{H}^{\ga-1}}\right. \\
	&{}\qquad \left. + \norm{u}_{L^{\infty}_T \dot{H}^{2mj-1+\ga-1}} + \norm{u}_{L^{\infty}_T \dot{H}^{2mj-m+\ga-1}} \) \\
	\le{}& C \( \TxLebn{\Jbr{x}^{m-1} \pa_x^{\ga-1}u}{\I}{2} + \norm{u}_{L^{\infty}_T \dot{H}^{2mj-1+\ga-1}} \),
\end{align*}
which yields
\begin{align*}
	\ti{A}_0 \le C\la^{\a-1} \norm{u}_{X_T}^2.
\end{align*}
As for $1 \le k \le \ga$, the elements in $\ti{A}_k$ can be written as
\begin{align*}
	&{}\Jbr{x}^m \pa_x^k (|u|^{\a}) \pa_x^{2j-1+\ga-k} u \\
	={}& \sum_{n=1}^{k} \sum_{\substack{ \b_1 + \cdots + \b_n = k \\ 1 \le \beta_1, \cdots, \beta_n \le k}} C_{\vec{\beta}} \Jbr{x}^m |u|^{\alpha-2n}u^n \, \pa_x^{\beta_1}u \cdots \partial_x^{\beta_n}u\, \pa_x^{\b_0} u \\
	=:{}& \sum_{n=1}^{k} \sum_{\substack{ \b_1 + \cdots + \b_n = k \\ 1 \le \beta_1, \cdots, \beta_n \le k}} \ti{F}_{k},
\end{align*}
where $\b_0 = 2j-1+\ga-k$. In the case $k \le \ga -4$ and $\ga \ge 5$, noting that $\b_0 > 2j+2$, using $\Jbr{x}^m |u(t, x)| \ge \frac{\l}{2}$ and Lemma \ref{lem:0}, together with the H\"older and the Gagliardo-Nirenberg inequality, we are led to
\begin{align*}
\TxLebn{\ti{F}_k}{\I}{2} \le{}& C\l^{\a-n} \TxLebn{\Jbr{x}^{m(n+1-\a)} \pa_x^{\beta_1}u \cdots \partial_x^{\beta_n}u\, \pa_x^{\b_0} u}{\I}{2} \\
	\le{}& C \l^{\a-n} \prod_{i=1}^{n} \TxLebn{\Jbr{x}^{m}\, \pa_x^{\b_i}u}{\I}{\I} \TxLebn{\Jbr{x}^{m-1}\, \pa_x^{\b_0}u}{\I}{2}  \\
		\le{}& C \l^{\a-n} \prod_{i=1}^{n} \( \TxLebn{\Jbr{x}^{m}\, \pa_x^{\b_i}u}{\I}{2} + \TxLebn{\Jbr{x}^{m}\, \pa_x^{\b_i+1}u}{\I}{2} \) \\
	&{} \times \( \TxLebn{\Jbr{x}^{m} \pa_x^{\ga-1-k}u}{\I}{2} + \TxLebn{\pa_x^{2mj+\ga-1-k}u}{\I}{2} + L.O.T. \) \\
	\le{}& C\l^{\a-n} \norm{u}_{X_T}^{n+1}.
\end{align*}
One easily verifies the same estimate in the other cases $\ga -3 \le k$ or $\ga \le 4$, because of $\b_i \le 2j+2$ for any $0 \le i \le n$. Indeed, if $k=\ga$, we then calculate that
\begin{align*}
	&{}\TxLebn{\Jbr{x}^m \pa_x^\ga (|u|^{\a}) \pa_x^{2j-1} u}{\I}{2} \\
	\le{}& C \la^{\a-1} \TxLebn{\Jbr{x}^{m(2-\a)} |\pa_x^{\ga} u| \pa_x^{2j-1} u}{\I}{2} \\
	&{}+ C \la^{\a-n} \sum_{n=2}^{\ga} \sum_{\substack{ \b_1 + \cdots + \b_n = \ga \\ 1 \le \beta_1, \cdots, \beta_n \le \ga-1}} \TxLebn{\Jbr{x}^{m(n+1-\a)} \pa_x^{\beta_1}u \cdots \partial_x^{\beta_n}u\, \pa_x^{2j-1} u}{\I}{2}.
\end{align*}
For the first term, it follows from Sobolev embedding that 
\begin{align*}
	&{}\TxLebn{\Jbr{x}^{m(2-\a)} |\pa_x^{\ga} u| \pa_x^{2j-1} u}{\I}{2} \\
	\le{}& C \( \TxLebn{\Jbr{x}^{m} \pa_x^{2j-1} u }{\I}{2} + \TxLebn{\Jbr{x}^{m} \pa_x^{2j} u }{\I}{2} \) \TxLebn{\Jbr{x}^{m} \pa_x^{\ga} u}{\I}{2} \\
	\le{}& C \norm{u}_{X_T}^2.
\end{align*}
In the similar way, the other terms are estimated as follows:
\begin{align*}
	&{}\TxLebn{\Jbr{x}^{m(n+1-\a)} \pa_x^{\beta_1}u \cdots \partial_x^{\beta_n}u\, \pa_x^{2j-1} u}{\I}{2} \\
	\le{}& \prod_{i=1}^{n} \( \TxLebn{\Jbr{x}^{m} \pa_x^{\be_i} u }{\I}{2} + \TxLebn{\Jbr{x}^{m} \pa_x^{\be_i+1} u }{\I}{2} \) \TxLebn{\Jbr{x}^m \pa_x^{2j-1} u}{\I}{2} \\
	\le{}& C\norm{u}_{X_T}^{n+1}.
\end{align*}
Combining these estimates, we find
\begin{align*}
	\ti{A}_k \le C \sum_{n=1}^k \l^{\a-n} \norm{u}_{X_T}^{n+1} \le C \l^{\a-1} \norm{u}_{X_T}^{2} + C \l^{\a-k} \norm{u}_{X_T}^{k+1}
\end{align*}
for any $1 \le k \le \ga$. These yield \eqref{lem:43}. This completes the proof.

\end{proof}

\begin{lemma} 
Let $m= \left[ \frac{1}{\a} \right]+1$. Denote $s$, $j  \in \Z^+$ with $s-j+1 \in \Z^+$. Put $q \in (1, 2]$ and $\ga \in \Z^+$ with $\ga \le 2j+2$. Then it holds that
\begin{align}
&{}	\begin{aligned}
	&{}\TxLebn{|u|^{\a} \pa_x^{2j-1} u - |v|^{\a} \pa_x^{2j-1} v}{1}{2} \\
	\le{}& CT \norm{u}_{X_T}^{\a} \norm{u-v}_{X_T}+C\l^{\a-1} T \norm{u}_{X_T} \norm{u-v}_{X_T},
	\end{aligned}
	\label{lem:50} \\
&{}	\begin{aligned}
	&{}\xTLebn{\partial_x^{s-j+1} (|u|^{\a} \pa_x^{2j-1} u - |v|^{\a} \pa_x^{2j-1} v)}{q}{2} \\
	\le{}& C \norm{u}_{X_T}^{\a} \norm{u-v}_{X_T} + C \l^{\a-1} \norm{u}_{X_T} \norm{u-v}_{X_T} \\
	&{}+ CT^{1/2} \( \l^{\a-(s-j+2)} + \l^{\a-(s-j+1)} \) \\
	&{} \quad \times \(\norm{u}_{X_T}^{s-j+2} + \norm{u}_{X_T}^{s-j+1} \) \norm{u-v}_{X_T} \\
	&{}+ CT^{1/2} \( \l^{\a-2} + \l^{\a-1} \) \( \norm{u}_{X_T}^{2} + \norm{u}_{X_T} \) \norm{u-v}_{X_T},
	\end{aligned}
\label{lem:51} \\
&{}	\begin{aligned}
	&{}\TxLebn{\Jbr{x}^m \pa_x^{\ga} \( |u|^{\a}\pa_x^{2j-1} u - |v|^{\a}\pa_x^{2j-1} v \)}{1}{2} \\
	\le{}& CT\la^{\a-1} \norm{u}_{X_T} \norm{u-v}_{X_T} \\
	&{}+ CT \( \la^{\a-\ga-1} + \la^{\a-\ga} \) \( \norm{u}_{X_T}^{\ga+1} + \norm{u}_{X_T}^{\ga} \) \norm{u-v}_{X_T} \\
	&{}+ CT \la^{\a-2} \norm{u}_{X_T}^2 \norm{u-v}_{X_T},
	\end{aligned}
\label{lem:52}
\end{align}
as long as $\Jbr{x}^m \min \( |u(t, x)|, |v(t, x)| \) \ge \frac{\l}{2}$ for any $(t,x) \in [0,T] \times \R$.
\end{lemma}

\begin{proof}
For simplicity, we shall only consider real-valued functions. 
The proof  is similar to that of Lemma \ref{lem:4} except for employing
\begin{align}\label{di:1}
	\left| |u|^{\a-2k}u^k - |v|^{\a-2k}v^k \right| \le C \( |u|^{\a-k-1} + |v|^{\a-k-1} \)|u-v| 
\end{align}
for any $k \in \Z_{\ge 0}$, so we shall only sketch the proof of \eqref{lem:51}.
By the Leibniz rule, one has
\begin{align*}
	&{}\xTLebn{\partial_x^{s-j+1} \( |u|^{\a} \pa_x^{2j-1} u - |v|^{\a} \pa_x^{2j-1} v \)}{q}{2} \\
	\le{}& 	\sum_{k=0}^{s-j+1} \xTLebn{\pa_x^k (|u|^{\a}) \pa_x^{s+j-k} u - \pa_x^k (|v|^{\a}) \pa_x^{s+j-k} v}{q}{2} =: \sum_{k=0}^{s-j+1} B_k.
\end{align*}
Combining the H\"older inequality with $\Jbr{x}^m \min \( |u|, |v| \) \ge \frac{\l}{2}$ and \eqref{di:1}, it is deduced that
\begin{align*} 
	B_0 
	\le{}& C \xTLebn{|u|^{\a} \pa_x^{s+j}(u -v)}{q}{2} + C\xTLebn{(|u|^{\a-1} + |u|^{\a-1})|u-v| \pa_x^{s+j}v}{q}{2} \\
	\le{}& C \TxLebn{\Jbr{x}^m u}{\I}{\I}^{\a} \xTLebn{\pa_x^{s+j}(u-v)}{\I}{2}  \\
	&+ C \l^{\a-1} \TxLebn{ \Jbr{x}^m (u-v)}{\I}{\I}  \xTLebn{\pa_x^{s+j}v}{\I}{2}.
\end{align*}
The estimation of $B_k$ for $1 \le k \le s-j+1$ can be obtained in the similar way to the proof of \eqref{lem:42}. Hence we here only handle $B_{s-j+1}$.
Thanks to the H\"older inequality, one sees that   
\begin{align*}
	&{}B_{s-j+1} 
	\le CT^{1/2} \( \left\| |u|^{\a-2(s-j+1)} u^{s-j+1} \( \pa_x u \)^{s-j+1} \pa_x^{2j-1} u \right. \right. \\
	&{} \qquad \qquad \left. - |v|^{\a-2(s-j+1)} v^{s-j+1} \( \pa_x v \)^{s-j+1} \pa_x^{2j-1} v \right\|_{L^\I_T L^{q}_x} + \cdots \\
	&{}\quad \left. \cdots + \TxLebn{|u|^{\a-2}u \( \pa_x^{s-j+1} u\) \pa_x^{2j-1} u - |v|^{\a-2}v \( \pa_x^{s-j+1} v\) \pa_x^{2j-1} v}{\I}{q} \) \\
	=:{}& CT^{1/2} \( B_{s-j+1,1} + \cdots + B_{s-j+1, s-j+1} \).
\end{align*}
As in the proof of \eqref{lem:42}, the estimate of $B_{s-j+1, k}$ for $2 \le k \le s-j+1$ can be carried out, so we only give a proof of $B_{s-j+1, 1}$ and $B_{s-j+1, s-j+1}$.
By using the Gagliardo-Nirenberg inequality, \eqref{di:1} and $\Jbr{x}^m \min \( |u|, |v| \) \ge \frac{\l}{2}$, one has
\begin{align*}
	&{}B_{s-j+1,1} \\
	\le{}& \TxLebn{\( |u|^{\a-2(s-j+1)} u^{s-j+1} - |v|^{\a-2(s-j+1)} v^{s-j+1} \) \( \pa_x u \)^{s-j+1} \pa_x^{2j-1} u}{\I}{q} \\
	&{} + \TxLebn{|v|^{\a-(s-j+1)} \( \( \pa_x u \)^{s-j+1} - \( \pa_x v \)^{s-j+1} \) \pa_x^{2j-1} u}{\I}{q} \\
	&{} + \TxLebn{|v|^{\a-(s-j+1)} \( \pa_x v \)^{s-j+1} \pa_x^{2j-1} (u-v) }{\I}{q} \\
	\le{}& C \l^{\a-(s-j+2)} \( \TxLebn{\Jbr{x}^{m} \pa_x u}{\I}{2} + \TxLebn{\Jbr{x}^{m} \pa_x^2 u}{\I}{2} \)^{s-j+1} \\
	&{} \quad \times \TxLebn{\Jbr{x} (u-v) }{\I}{\I} \TxLebn{\pa_x^{2j-1} u}{\I}{2} \\
	&{}+ C \l^{\a-(s-j+1)} \( \TxLebn{\Jbr{x}^{m} \pa_x u}{\I}{2} + \TxLebn{\Jbr{x}^{m} \pa_x^2 u}{\I}{2} \right. \\
	&{} \left. \hspace{3.5cm} + \TxLebn{\Jbr{x}^{m} \pa_x v}{\I}{2} + \TxLebn{\Jbr{x}^{m} \pa_x^2 v}{\I}{2} \)^{s-j} \\
	&{} \quad \times \( \TxLebn{\Jbr{x}^{m} \pa_x (u-v)}{\I}{2} + \TxLebn{\Jbr{x}^{m} \pa_x^2 (u-v)}{\I}{2} \) \TxLebn{\pa_x^{2j-1} u}{\I}{2} \\
	&{}+ C \l^{\a-(s-j+1)} \( \TxLebn{\Jbr{x}^{m} \pa_x v}{\I}{2} + \TxLebn{\Jbr{x}^{m} \pa_x^2 v}{\I}{2} \)^{s-j+1} \\
	&{} \quad \times \TxLebn{\pa_x^{2j-1} (u-v)}{\I}{2}.
\end{align*}
Similarly to the above,
\begin{align*}
	&{}B_{s-j+1, s-j+1} \\
	\le {}& C \la^{\a-2} \TxLebn{\Jbr{x}^m (u-v)}{\I}{\I} \\
	&{} \quad \times \( \TxLebn{\Jbr{x}^{m} \pa_x^{2j-1} u}{\I}{2} + \TxLebn{\Jbr{x}^{m}\pa_x^{2j} u}{\I}{2} \) \TxLebn{\pa_x^{s-j+1} u}{\I}{2} \\
	&{}+ C \la^{\a-1} \( \TxLebn{\Jbr{x}^{m} \pa_x^{2j-1} u}{\I}{2} + \TxLebn{\Jbr{x}^{m}\pa_x^{2j} u}{\I}{2} \) \\
	&{} \quad \times \TxLebn{\pa_x^{s-j+1} (u-v)}{\I}{2} \\
	&{}+ C \la^{\a-1} \TxLebn{\pa_x^{s-j+1} v}{\I}{2} \\
	&{} \quad \times \( \TxLebn{\Jbr{x}^{m} \pa_x^{2j-1} (u-v)}{\I}{2} + \TxLebn{\Jbr{x}^{m}\pa_x^{2j} (u-v)}{\I}{2} \).
\end{align*}
Combining these estimates, we reach to \eqref{lem:51}. 

\end{proof}

\section{Proof of the main results. } \label{sec:4}

\begin{proof}[Proof of Theorem \ref{thm:1}]
To simplify the exposition, we shall only consider real-valued functions. Let us introduce the complete metric space
\begin{align*}
	X_{T,M} = {}&\left\{ u \in C([0,T], H^s(\R)); \right.
	\norm{u}_{X_{T}} := \norm{u}_{L^{\I}_T H^{s}} + \norm{\Jbr{x}^{m} u}_{L^{\I}_T L^\I_x} \\
	&{} \qquad + \sum_{\ga =1}^{2j+2} \norm{\Jbr{x}^{m} \pa^{\ga}_{x} u}_{L^{\I}_T L^2_x} + \sum_{l=0}^{j-1} \xTLebn{\pa_x^{s+j-l} u}{\I}{2} \le M, \\
	{}& \qquad \left. \sup_{0 \le t \le T} \Lebn{\Jbr{x}^m (u(t)-u_0)}{\I} \le \frac{\l}2 \right\}
\end{align*}
equipped with the distance function $d_{X_{T}}(u,v) = \norm{u-v}_{X_{T}}$.
Here the constant $M$ will be chosen later.
Notice that
\begin{align}
	\frac{\l}{2} \le \Jbr{x}^m |u(x,t)| \le \Jbr{x}^m |u_0(x)| + \frac{\l}{2} \label{cond:1}
\end{align}
for any $(x,t) \in \R \times [0,T]$ as long as $u \in X_{T, M}$.
Set 
\begin{align*}
	\P(u(t)) = U_j(t)u_0 \mp \int_0^t U_j(t-s)(|u|^{\a}\pa_x^{2j-1} u)(s) ds. 
\end{align*}
We will prove that $\P$ is a contraction map in $X_{T, M}$. 
Let us first show that $\P$ maps from $X_{T, M}$ to itself. 
By Lemma \ref{lem:2}, it holds that
\begin{align*}
	&{}\TxLebn{\pa_x^{s} \P(u)}{\I}{2} + \sum_{l=0}^{j-1} \xTLebn{\pa_x^{s+j-l} \P(u)}{\I}{2} \\
	\le {}&  C \sum_{l=0}^{j-1} \Lebn{\pa_x^{s-l} u_0}{2} + C \sum_{l=0}^{j-1} T^{\a_{j, l}} \xTLebn{\partial_x^{s-j+1} (|u|^{\a} \pa_x^{2j-1} u)}{p_{j, l}}{2} \\
	\le {}&  C \Sobn{u_0}{s} + C \( T^{\a_{j, 0}} + T^{\a_{j, j-1}} \) \\
	&{} \times \( \xTLebn{\partial_x^{s-j+1} (|u|^{\a} \pa_x^{2j-1} u)}{p_{j,0}}{2} \right. \\
	&{} \left. \hspace{3cm} + \xTLebn{\partial_x^{s-j+1} (|u|^{\a} \pa_x^{2j-1} u)}{p_{j,j-1}}{2}\),
\end{align*}
where $\a_{j, l} = \frac{l+1}{2j}$ and $p_{j, l} = \frac{2j}{2j-1-l}$.
By employing \eqref{lem:42} with $q= p_{j, 0}$ and $p_{j, j-1}$, we obtain
\begin{align}
	\begin{aligned}
	&{}\TxLebn{\pa_x^{s} \P(u)}{\I}{2} + \sum_{l=0}^{j-1} \xTLebn{\pa_x^{s+j-l} \P(u)}{\I}{2} \\
	\le{}& C \Sobn{u_0}{s} + CT^{\frac{1}{2j}} \( 1 + \la^{\a-(s-j+1)} \)(M^{\a+1} + M^{s-j+2})
	\end{aligned}
	\label{est:4}
\end{align} 
as long as $T \le 1$. One also sees from \eqref{lem:40} that
\begin{align}
	\TxLebn{\P(u)}{\I}{2} \le{}& \Lebn{u_0}{2} + CTM^{\a+1}. \label{est:2}
\end{align}

Let us next consider $\norm{\Jbr{x}^{m} \pa^{\ga}_x \P(u)}_{L^{\I}_T L^2_x}$ for any $1 \le \ga \le 2j+2$. It follows from Lemma \ref{lem:3} that
\begin{align}
\begin{aligned}
	&\TxLebn{\Jbr{x}^{m} \pa^{\ga}_x \P(u)}{\I}{2}\\
	\le{}& C \Lebn{\Jbr{x}^m \pa_x^{\ga} u_0}{2} \\
	&{}+ C \Jbr{T}^{m} \( \Lebn{\pa_x^{\ga} u_0}{2} + \Lebn{\pa_x^{2jm+{\ga}}u_0}{2} \) \\
	&{}+ C \TxLebn{\Jbr{x}^m \pa_x^{\ga} (|u|^{\a}\pa_x^{2j-1} u)}{1}{2} \\
	&{}+ C \Jbr{T}^m \( \TxLebn{\pa_x^{\ga} \( |u|^{\a}\pa_x^{2j-1} u \)}{1}{2} \right. \\
	&{} \hspace{3cm} \left. + \TxLebn{\pa_x^{2jm+\ga} \( |u|^{\a}\pa_x^{2j-1} u \)}{1}{2} \). 
\end{aligned}
	\label{esn:1}
\end{align}
Since $s - j +1 \ge 2jm+2j+2$, we deduce from $H^{s-j+1} \hookrightarrow \dot{H}^{\s}$ for $\s = \ga$, $2jm+\ga$ that
\begin{align*}
	&{}\TxLebn{\pa_x^{\ga} (|u|^{\a}\pa_x^{2j-1} u)}{1}{2} + \TxLebn{\pa_x^{2jm+\ga} (|u|^{\a}\pa_x^{2j-1} u)}{1}{2} \\
	\le{}& C \TxLebn{|u|^{\a}\pa_x^{2j-1} u}{1}{2} + CT^{1/2} \xTLebn{\pa_x^{s-j+1} (|u|^{\a}\pa_x^{2j-1} u)}{2}{2}.
\end{align*}
Applying \eqref{lem:40} and \eqref{lem:42} with $q = 2$, one has
\begin{align}
\begin{aligned}
	&{}\TxLebn{\pa_x^{\ga} (|u|^{\a}\pa_x^{2j-1} u)}{1}{2} + \TxLebn{\pa_x^{2jm+\ga} (|u|^{\a}\pa_x^{2j-1} u)}{1}{2} \\
	\le {}& CT^{1/2}\(1 + \la^{\a-(s-j+1)} \)\( M^{\a+1} + M^{s-j+2} \)
\end{aligned}
\label{est:in}
\end{align}
for any $T \le 1$. Further, it comes from \eqref{lem:43} that
\begin{align}
	\TxLebn{\Jbr{x}^m \pa_x^{\ga} (|u|^{\a}\pa_x^{2j-1} u)}{1}{2} \le CT\la^{\a-1} M^2 + CT\la^{\a-\ga}M^{\ga+1}. 
	\label{esn:2}
\end{align}
Hence, it is concluded from \eqref{esn:1}, \eqref{est:in} and \eqref{esn:2} that
\begin{align}
	\begin{aligned}
	&{}\norm{\Jbr{x}^{m} \pa^{\ga}_x \P(u)}_{L^{\I}_T L^2_x} \\
	\le{}& C \( \Lebn{\Jbr{x}^m \pa_x^{\ga} u_0}{2} +  \Lebn{\pa_x^{\ga} u_0}{2} + \Lebn{\pa_x^{2jm+\ga}u_0}{2} \) \\
	&{}+ CT^{1/2} \(1 + \la^{\a-(s-j+1)} \)\( M^{\a+1} + M^{s-j+2} \)
	\end{aligned}
	\label{est:3}
\end{align}
as long as $T \le 1$ for any $1 \le \ga \le 2j+2$. 

Let us next handle $\TxLebn{\Jbr{x}^m \P(u)}{\I}{\I}$.
Thanks to 
\[
\frac{d}{dt} U_j(t)u_0 = -\pa_x^{2j+1} U_j(t) u_0,
\]
combining Sobolev embedding with Lemma \ref{lem:3}, we obtain
\begin{align}
	\begin{aligned}
	{}&\TxLebn{\Jbr{x}^m (U_j(t)u_0 -u_0)}{\I}{\I} \\
	\le {}& \TxLebn{\int_0^t \Jbr{x}^m U_j(s) \pa_x^{2j+1} u_0\, ds}{\I}{\I} \\
	\le {}& CT \( \Lebn{\Jbr{x}^{m} \pa_x^{2j+1} u_0}{2} + \Lebn{\Jbr{x}^{m} \pa_x^{2j+2} u_0}{2} +\Sobn{u_0}{2jm+2j+2} \)
	\end{aligned}
	\label{est:5}
\end{align}
whenever $T \le 1$. Similarly to \eqref{est:5}, one has
\begin{align*}
	&{}\TxLebn{\Jbr{x}^m \int_0^t U_j(t-s) \( |u|^{\a}\pa_x^{2j-1} u\)(s) ds}{\I}{\I} \\
	\le {}& C\Jbr{T}^{m} \( \TxLebn{\Jbr{x}^{m} |u|^{\a}\pa_x^{2j-1} u}{1}{2} + \TxLebn{\Jbr{x}^{m} \pa_x \(|u|^{\a}\pa_x^{2j-1} u\)}{1}{2} \) \\
	&{} + C\Jbr{T}^{m} \( \TxLebn{|u|^{\a}\pa_x^{2j-1} u}{1}{2} + \TxLebn{ \pa_x^{2jm+1} \(|u|^{\a}\pa_x^{2j-1} u\)}{1}{2} \). 
\end{align*}
By means of \eqref{cond:1}, we here compute 
\begin{align*}
	\TxLebn{\Jbr{x}^{m} |u|^{\a}\pa_x^{2j-1} u}{1}{2} \le{}& T \TxLebn{\Jbr{x}^m u}{\I}{\I} \TxLebn{|u|^{\a-1} \pa_x^{2j-1} u}{\I}{2} \\
	\le{}& CT\la^{\a-1} \TxLebn{\Jbr{x}^m u}{\I}{\I} \TxLebn{\Jbr{x}^{m} \pa_x^{2j-1} u}{\I}{2} \\
	\le{}& CT\la^{\a-1} M^2.
\end{align*}
By \eqref{lem:43} with $\ga=1$, it is observed that
\begin{align*}
	\TxLebn{\Jbr{x}^{m} \pa_x (|u|^{\a}\pa_x^{2j-1} u)}{1}{2} 
	\le{}& CT \la^{\a-1}M^2.
\end{align*}
Further, arguing as in \eqref{est:in}, we deduce that 
\begin{align*}
	&{}\TxLebn{|u|^{\a}\pa_x^{2j-1} u}{1}{2} + \TxLebn{ \pa_x^{2jm+1} \(|u|^{\a}\pa_x^{2j-1} u\)}{1}{2} \\
	\le{}& CT^{1/2} \(1 + \la^{\a-(s-j+1)} \)\( M^{\a+1} + M^{s-j+2} \)
\end{align*}
for any $T \le 1$. Therefore it follows from these estimates that
\begin{align}
	\begin{aligned}
	&{}\TxLebn{\Jbr{x}^m \int_0^t U(t-s) \( |u|^{\a}\pa_x^{2j-1} u\)(s) ds}{\I}{\I} \\
	\le{}& CT^{1/2} \(1 + \la^{\a-(s-j+1)} \)\( M^{\a+1} + M^{s-j+2} \)
	\end{aligned}
	\label{est:6}
\end{align}
as long as $T \le 1$. Combining \eqref{est:5} with \eqref{est:6}, we see that
\begin{align}
	\begin{aligned}
	&{}\txLebn{\Jbr{x}^m \P(u)}{\I}{\I} \\
	\le{}& \Lebn{\Jbr{x}^{m}u_0}{\I} \\
	&{}+ CT \( \Lebn{\Jbr{x}^{m} \pa_x^{2j+1} u_0}{2} + \Lebn{\Jbr{x}^{m} \pa_x^{2j+2} u_0}{2} +\Sobn{u_0}{2jm+2j+2} \) \\
	&{}+ CT^{1/2} \(1 + \la^{\a-(s-j+1)} \)\( M^{\a+1} + M^{s-j+2} \)
	\end{aligned}
	\label{est:7}
\end{align}
whenever $T \le 1$.
By collecting \eqref{est:4}, \eqref{est:2}, \eqref{est:3}, and \eqref{est:7}, it is established that
\begin{align*}
	\norm{u}_{X_T} 
	\le{}& C_0(1+T) \d + CT^{\frac1{2j}} \(1 + \la^{\a-(s-j+1)} \)\( M^{\a+1} + M^{s-j+2} \) \\
	\le{}& C_1 \d + CT^{\frac1{2j}} \( \frac{1+\la}{\la} \)^{s-j+1-\a} (M^{\a+1} + M^{s-j+2}) 
\end{align*}
for any $T \le 1$, where $C_1 = 2C_0$. Therefore, taking $M = 2C_1 \d$, we have $\norm{u}_{X_T} \le M$
whenever $T = T(\d, \la; \a, s, j)$ satisfies
\begin{align}
	CT^{\frac1{2j}} \( \frac{1+\la}{\la} \)^{s-j+1-\a} \( \d^{\a} + \d^{s-j+1} \) \le 1. \label{ls:1}
\end{align}
Note that in view of $\d \ge \la$, $T \le 1$ holds as long as $T$ satisfies \eqref{ls:1}. 
Finally, combining \eqref{est:5} with \eqref{est:6}, when $T \le 1$, we easily have
\begin{align*}
	&{}\Lebn{\Jbr{x}^m (\P(u(t))-u_0)}{\I} \\
	\le{}& CT^{1/2}\d + CT^{1/2} \( \frac{1+\la}{\la} \)^{s-j+1-\a} (M^{\a+1} + M^{s-j+2}),
\end{align*}
which implies 
\[
	\sup_{0 \le t \le T} \Lebn{\Jbr{x}^m (\P(u(t))-u_0)}{\I} \le \frac{\l}{2}
\]
as long as $T = T(\d, \la; \a, s)$ satisfies
\begin{align}
	CT^{1/2} \d + CT^{1/2} \d \( \frac{1+\la}{\la} \)^{s-j+1-\a} ( \d^{\a} + \d^{s-j+1}) \le \frac{\l}{2} \label{time:2}
\end{align}
since $M = 2C_1 \d$.
Thus taking $T = T(\d, \la; \a, s, j)$ satisfying \eqref{ls:1} and \eqref{time:2}, $\P(u) \in X_{T,M}$ is valid.

Let us show $\P$ is a contraction map in $X_{T, M}$. The proof is very similar to the above proof, so we sketch the proof.
Arguing as in the proof of \eqref{est:4}, 
we deduce from \eqref{lem:51} that
\begin{align}
	\begin{aligned}
	&{}\TxLebn{\pa_x^{s} \( \P(u)-\P(v)\)}{\I}{2} + \sum_{l=0}^{j-1} \xTLebn{\pa_x^{s+j-l} \( \P(u)-\P(v)\)}{\I}{2} \\
	\le{}& CT^{\frac{1}{2j}} \( 1 + \la^{\a-(s-j+2)} \)(M^{\a} + M^{s-j+2}) d_{X_{T}}(u,v)
	\end{aligned}
	\label{estd:4}
\end{align} 
as long as $T \le 1$. Also, \eqref{lem:50} gives us 
\begin{align}
	\TxLebn{\P(u) - \P(v)}{\I}{2} \le{}& CT(M^{\a} +M) d_{X_{T}}(u,v). \label{estd:2}
\end{align}

We will turn to treat
\[
\norm{\Jbr{x}^{m} \pa^{\ga}_x \( \P(u) - \P(v) \)}_{L^{\I}_T L^2_x}
\] 
for any $1 \le \ga \le 2j+2$. It follows from Lemma \ref{lem:3} that
\begin{align}
\begin{aligned}
	&\TxLebn{\Jbr{x}^{m} \pa^{\ga}_x \( \P(u) - \P(v) \)}{\I}{2}\\
	\le{}& C \Jbr{T}^{m} \TxLebn{\Jbr{x}^m \pa_x^{\ga} (|u|^{\a}\pa_x^{2j-1} u - |v|^{\a}\pa_x^{2j-1} v)}{1}{2} \\
	&{}+ C \Jbr{T}^m \( \TxLebn{\pa_x^{\ga} \( |u|^{\a}\pa_x^{2j-1} u - |v|^{\a}\pa_x^{2j-1} v\)}{1}{2} \right. \\
	&{} \hspace{3cm} \left. + \TxLebn{\pa_x^{2jm+\ga} \( |u|^{\a}\pa_x^{2j-1} u  - |v|^{\a}\pa_x^{2j-1} v \)}{1}{2} \). 
\end{aligned}
	\label{esnd:1}
\end{align}
As for the first term of R.H.S in \eqref{esnd:1}, it follows from \eqref{lem:52} that
\begin{align}
	\begin{aligned}
	&{}\TxLebn{\Jbr{x}^m \pa_x^{\ga} \( |u|^{\a}\pa_x u - |v|^{\a}\pa_x v \)}{1}{2} \\
	\le{}& CT\( \la^{\a-2} + \la^{\a-1} \) (M + M^2) d_{X_{T}}(u,v) \\
	&{}+ CT\( \la^{\a-\ga} + \la^{\a-(\ga+1)} \) \( M^{\ga} + M^{\ga+1} \) d_{X_{T}}(u,v). 
	\end{aligned}
	\label{esnd:2}
\end{align}
Applying \eqref{lem:50} and \eqref{lem:51} with $q = 2$, together with $H^{s-j+1} \hookrightarrow \dot{H}^{\s}$ for $\s = \ga$, $2jm+\ga$, the last term can be estimated as follows:
\begin{align}
\begin{aligned}
&{}\TxLebn{\pa_x^{\ga} (|u|^{\a}\pa_x^{2j-1} u - |v|^{\a}\pa_x^{2j-1} v)}{1}{2} \\
	&{}+ \TxLebn{\pa_x^{2jm+\ga} \( |u|^{\a}\pa_x^{2j-1} u  - |v|^{\a}\pa_x^{2j-1} v \)}{1}{2} \\
	\le {}& CT^{1/2} \( 1 + \la^{\a-(s-j+2)} \)(M^{\a} + M^{s-j+2})d_{X_{T}}(u,v)
\end{aligned}
\label{estd:in}
\end{align}
for any $T \le 1$. 
Hence, we conclude from \eqref{esnd:1}, \eqref{esnd:2} and \eqref{estd:in} that
\begin{align}
	\begin{aligned}
	&{}\norm{\Jbr{x}^{m} \pa^{\ga}_x \( \P(u) - \P(u) \)}_{L^{\I}_T L^2_x} \\
	\le{}& CT^{1/2} \(1 + \la^{\a-(s-j+2)} \)\( M^{\a} + M^{s-j+2} \)d_{X_{T}}(u,v)
	\end{aligned}
	\label{estd:3}
\end{align}
as long as $T \le 1$ for any $1 \le \ga \le 2j+2$. 
Let us next estimate 
\[
	\TxLebn{\Jbr{x}^m \( \P(u)-\P(v)\)}{\I}{\I}.
\]
Arguing as in \eqref{est:6}, by Sobolev embedding and Lemma \ref{lem:3},
we are led to
\begin{align*}
	&{}\TxLebn{\Jbr{x}^m (\P(u)-\P(v))}{\I}{\I} \\
	\le{}& CT \Jbr{T}^{m} \TxLebn{\Jbr{x}^m (|u|^{\a}\pa_x^{2j-1} u - |v|^{\a}\pa_x^{2j-1} v)}{1}{2} \\
	&{}+ CT \Jbr{T}^{m} \TxLebn{\Jbr{x}^m \pa_x (|u|^{\a}\pa_x^{2j-1} u - |v|^{\a}\pa_x^{2j-1} v)}{1}{2} \\
	&{}+ C\Jbr{T}^{m} \TxLebn{|u|^{\a}\pa_x^{2j-1} u - |v|^{\a}\pa_x^{2j-1} v}{1}{2} \\
	&{}+ C\Jbr{T}^m \TxLebn{\pa_x^{2jm+1} \(|u|^{\a}\pa_x^{2j-1} u - |v|^{\a}\pa_x^{2j-1} v \)}{1}{2}.
\end{align*}
Thanks to \eqref{di:1} and \eqref{cond:1}, The first term of R.H.S in the above can be estimated as follows:
\begin{align*}
	&{}\TxLebn{\Jbr{x}^{m} \( |u|^{\a}\pa_x^{2j-1} u - |v|^{\a}\pa_x^{2j-1} v\)}{1}{2} \\
	\le{}& CT \la^{\a-1} \TxLebn{\Jbr{x}^m (u-v)}{\I}{\I} \TxLebn{\Jbr{x}^{m} \pa_x^{2j-1} u}{\I}{2} \\
	&{}+ CT \la^{\a-1} \TxLebn{\Jbr{x}^m u}{\I}{\I} \TxLebn{\Jbr{x}^{m} \pa_x^{2j-1} (u-v)}{\I}{2} \\
	\le{}& CT \la^{\a-1} M d_{X_{T}}(u,v).
\end{align*}
We then see from \eqref{lem:52} with $\ga =1$ from that
\begin{align*}
	&{}\TxLebn{\Jbr{x}^{m} \pa_x \( |u|^{\a}\pa_x^{2j-1} u - |v|^{\a}\pa_x^{2j-1} v\)}{\I}{2} \\
	\le{}& CT \( \l^{\a-1} + \l^{\a-2} \)(M + M^2) d_{X_{T}}(u,v).
\end{align*}
As in the proof of \eqref{estd:in}, one reaches to
\begin{align*}
	&{}\TxLebn{|u|^{\a}\pa_x^{2j-1} u - |v|^{\a}\pa_x^{2j-1} v}{1}{2} \\
	&{}+ \TxLebn{\pa_x^{2jm+1} \(|u|^{\a}\pa_x^{2j-1} u - |v|^{\a}\pa_x^{2j-1} v \)}{1}{2} \\
	\le{}& CT^{1/2} \( 1 + \la^{\a-(s-j+2)} \)(M^{\a} + M^{s-j+2})d_{X_{T}}(u,v)
\end{align*}
for any $T \le 1$. Therefore it is established from these estimates that
\begin{align}
	\begin{aligned}
	&{}\TxLebn{\Jbr{x}^m (\P(u)-\P(v))}{\I}{\I} \\
	\le{}& CT^{1/2} \(1 + \la^{\a-(s-j+2)} \)\( M^{\a} + M^{s-j+2} \) d_{X_{T}}(u,v)
	\end{aligned}
	\label{estd:6}
\end{align}
as long as $T \le 1$. 
In conclusion, combining \eqref{estd:4} with \eqref{estd:2}, \eqref{estd:3} and \eqref{estd:6}, we see that
\begin{align*}
	d_{X_T}(\P(u),\P(v)) 
	\le{}& CT^{\frac1{2j}} \(1 + \la^{\a-(s-j+2)} \)\( M^{\a} + M^{s-j+2} \) d_{X_T}(u,v) \\
	\le{}& CT^{\frac1{2j}} \( \frac{1+\la}{\la} \)^{s-j+2-\a} (M^{\a} + M^{s-j+2}) d_{X_T}(u,v) 
\end{align*}
for any $T \le 1$. Therefore, recalling $M = 2C_1 \d$, we have 
\begin{align*}
	d_{X_T}(\P(u),\P(v)) \le \frac12 d_{X_T}(u, v)
\end{align*}
whenever $T = T(\d, \la; \a, s, j)$ satisfies
\begin{align}
	CT^{\frac1{2j}} \( \frac{1+\la}{\la} \)^{s-j+2-\a} \( \d^{\a} + \d^{s-j+2} \) \le \frac12. \label{lsd:1}
\end{align}
Note that $T \le 1$ holds when $T$ satisfies \eqref{lsd:1}. 
This implies that $\P$ is a contraction map in $X_{T, M}$, that is, \eqref{hkdv} has a unique local solution in $X_{T, M}$. The remainder of the proof is standard, so we omit the detail.
\end{proof}

\begin{proof}[Proof of Corollary \ref{cor:1}]
Set
\[
	T_1^{\frac{1}{2j}} = \frac{C \l}{\d \( 1 + \( \d^{\a} + \d^{s-j+2} \) \( \frac{1+\la}{\la} \)^{s-j+2-\a} \)}
\]
for some $C>0$. Thanks to $\d^{\a} + \d^{s-j+1} \le 2 (\d^{\a} + \d^{s-j+2})$, since there exists $C \in (0,1)$ such that $T_1$ satisfies \eqref{ls:1}, \eqref{time:2} and \eqref{lsd:1}, it follows from Theorem \ref{thm:1} that \eqref{hkdv} has a unique solution in $X_{T_1, M}$. By the definition of $T_{\d, \l}$, we obtain $T_{\d, \l} \ge T_1$. This completes the proof.
\end{proof}

\subsection*{Acknowledgments} 
This work was supported by the Research Institute for Mathematical Sciences, a Joint Usage/Research Center located in Kyoto University.
The author is grateful to the referee for reading our manuscript carefully and giving useful suggestions.

\bibliographystyle{amsplain}

\end{document}